\documentclass[11pt]{amsart}

\usepackage{amscd,amsmath,latexsym,amsthm,amsfonts,amssymb,graphicx,color,geometry,soul}

\usepackage{enumerate}

\usepackage[dvipsnames]{xcolor}

\usepackage{hyperref}
\hypersetup{
colorlinks=true,
linkcolor=blue,
citecolor=cyan
}

\allowdisplaybreaks

\usepackage[norefs,nocites]{refcheck}

\usepackage{geometry}
\newtheorem{theorem}{Theorem}[section]
\newtheorem{proposition}[theorem]{Proposition}

\newtheorem{corollary}[theorem]{Corollary}
\newtheorem{example}[theorem]{Example}

\numberwithin{equation}{section}

\begin{document}


\title[A summability principle and applications]{A summability principle and applications}

\author[Albuquerque]{N. Albuquerque}
\address[N. Albuquerque]{Departamento de Matem\'{a}tica \newline\indent
Universidade Federal da Para\'{i}ba \newline\indent
Jo\~ao Pessoa - PB \newline\indent
58.051-900 (Brazil)}
\email{ngalbuquerque@mat.ufpb.br}

\author[Ara\'{u}jo]{G. Ara\'{u}jo}
\address[G. Ara\'{u}jo]{Departamento de Matem\'{a}tica \newline\indent
Universidade Estadual da Para\'{i}ba \newline\indent
Campina Grande - PB \newline\indent
58.429-500 (Brazil)}
\email{gdasaraujo@gmail.com}

\author[Rezende]{L. Rezende}
\address[L. Rezende]{Departamento de Matem\'{a}tica \newline\indent
Universidade Federal da Para\'{i}ba \newline\indent
Jo\~ao Pessoa - PB \newline\indent
58.051-900 (Brazil)}
\email{lirezendestos@gmail.com}

\author[Santos]{J. Santos}
\address[J. Santos]{Departamento de Matem\'{a}tica \newline\indent
Universidade Federal da Para\'{i}ba \newline\indent
Jo\~ao Pessoa - PB \newline\indent
58.051-900 (Brazil)}
\email{joedson.santos@academico.ufpb.br}




\begin{abstract}
This paper investigates summability principles for multilinear summing operators. The main result presents a novel inclusion theorem for a class of summing operators, which generalizes several classical results. As applications, we derive improved estimates for Hardy--Littlewood inequalities on multilinear forms and prove a Grothendieck-type coincidence result in anisotropic settings.
\end{abstract}

\maketitle

\section{Introduction}

Summing operators were introduced in the seminal work of A. Grothendieck \cite{resume} in 1953 and further developed by J. Lindenstrauss and A. Pe\l czy'nski \cite{lindpel} in 1968. In the 1980s, A. Pietsch systematized the linear theory in his foundational monograph \cite{pietsch-book} and subsequently extended it to the multilinear setting \cite{pietsch83}. Since then, the theory has become a central subject in Functional Analysis, leading to extensive research. We refer to the classical works \cite{DJT,dtp-book,pietsch-book} for a comprehensive treatment, and to \cite{abps-jfa,abps-israel,matos,dav,zal} (among others) for more recent developments.

%
%

Regularity phenomena that improve or preserve summability properties of operators are fundamental in Functional Analysis. Among these, inclusion theorems are a distinguished class of results and have been extensively studied (see, e.g., \cite{AR,bbbb,bjp,bot,reg,dav}). In the multilinear context, such results become significantly more delicate (cf. \cite{bjp,dav}).

The main contribution of this work is an anisotropic inclusion theorem for the class of $\Lambda$-summing operators (see Section~\ref{sec2}), which extends two of the main important multilinear summing classes: \emph{absolutely} and \emph{multiple} summing operators. For instance, a particular case of the main result we prove is an inclusion of the type
\[
\Pi_{\left(r; p \right)}^{\Lambda} \left(E_{1}, \dots, E_{m}; F\right)
\subset
\Pi_{\left(s; q \right)}^{\Lambda} \left( E_{1}, \dots, E_{m}; F\right)
\]
where \(\Lambda\) can be both classes \emph{absolutely} or \emph{multiple} summing operators, \(r,s,p,q\) are suitable parameters and, as usual, \(E_1,\dots,E_m,F\) stand for Banach spaces. The main result (Theorem \ref{inclusion_blocks}) yields several applications, including connections with the Bohnenblust--Hille and Hardy--Littlewood inequalities, as well as a Grothendieck-type theorem, on which we set conditions on the parameters $(s;q)$ in order to coincidence result for multiple summing operators occurs:
\[
\Pi_{(s;q)}^{\mathrm{ms}} \left(^{m}\ell_{1};\ell_{2}\right) = \mathcal{L}\left(^{m}\ell_{1};\ell_{2}\right).
\]

The paper is organized as follows. Section \ref{sec2} contains preliminary material and we briefly discuss about the $\Lambda$-summing class and its well behavior block cases. Section \ref{sec3} is devoted to the proof of the inclusion theorem in a block-structured setting. Sections \ref{sec4} presents applications to classical inequalities and a Grothendieck-type coincidence result.

\section{Preliminaries and Key Concepts} \label{sec2}

We recall briefly some basic concepts and results. Our notation is standard, as in most textbooks on Banach space theory and Functional Analysis; we refer, e.g, \cite{bpt-book,fabian}. Always, $m$ denotes a positive integer. $E,E_1,\dots,E_m,F$ shall denote Banach spaces over the field \(\mathbb{K}\), which will be \(\mathbb{R}\) the field of real scalars or \(\mathbb{C}\) the field of complex scalars. We denote by $E^{\ast}$ and $B_{E}$ the topological dual and the closed unit ball of $E$, respectively. For $p\in [1,\infty)$, the Banach space $\ell_{p}^{w}(E)$ of weakly $p$-summable sequences on $E$ is the space of all sequences $(x_{j})_{j \in \mathbb{N}} \in E^\mathbb{N}$ such that
\[
\|(x_j)_{j\in\mathbb{N}}\|_{w,p}:= \sup_{x^\ast\in B_{E^\ast}}\Vert \left( x^\ast \left(x_{j}\right) \right) _{j=1}^{\infty }\Vert _{p} = \sup_{x^\ast\in B_{E^\ast}} \left( \sum_{j=1}^{\infty}|x^\ast(x_j)|^p\right)^{\frac{1}{p}}<\infty,
\]
\(\mathcal{L}(E_1,\dots, E_m;F)\) stands for the Banach space of all bounded $m$-linear (multilinear if $m>1$) operators \(T: E_1 \times \cdots \times E_m \to F\) endowed with the usual sup norm. We refer the reader to \cite{dineen,mujica} for more details on the general theory of multilinear operators. In order to keep the notation as simple as possible in the multilinear framework, $\mathbf{i} := (i_1,\dots,i_m) \in \mathbb{N}^m$ shall denote a multi-index, and $\mathbf{r} := (r_1,\ldots,r_m) \in [1,+\infty]^m$ a multi-parameter. Also we will write $T x_{\mathbf{i}} := T(x_{i_1},\ldots,x_{i_m})$ with \(x_{i_k} \in E_k, \ k=1, \dots, m\).

\subsection{Multilinear summing operators: classical approach}

The very first definition of summing multilinear operators dates back to A. Pietsch in \cite{pietsch83}. In modern terminology the concept is defined as follows. For $r\geq 1$ and $\mathbf{p} := (p_1,\ldots,p_m) \in\lbrack1,\infty)^{m}$, a multilinear operator $T:E_{1} \times\cdots\times E_{m} \to F$ is \emph{absolutely} $(r;\mathbf{p})$-summing if there exists a constant $C>0$ such that, 
\begin{equation}\label{mult_def}
\left\|  \left(  T( x_{i}^{1},\dots,x_{i}^{m} ) \right)_{i \in\mathbb{N}}\right\|_{\ell_r(F)}
= \left(
    \sum_{i=1}^{\infty} \left\| T\left(x_{i}^{1}, \dots, x_{i}^{m}\right) \right\|^{r}
  \right)^{\frac{1}{r}}
\leq C \prod_{j=1}^{m} \|(x_i^{j})_{i \in \mathbb{N}} \|_{w,p_{k}},
\end{equation}
for all $(x_{i}^{j})_{i\in\mathbb{N}} \in\ell_{p_{j}}^{w}(E_{j}), \, j=1,\dots,m$.

The study of summing operators was extended to the multilinear setting following Pietsch's work \cite{pietsch83}. Subsequently, several distinct research directions developed. Among these, alongside the theory of absolutely summing operators, one particularly fruitful class emerged through independent work by M. Matos \cite{matos} and F. Bombal, D. P\'erez-Garc\'ia and I. Vilanueva \cite{bombal}: an operator $T: E_{1} \times \cdots \times E_{m} \to F$ is called \emph{multiple} $(r;\mathbf{p})$-summing if it satisfies property \eqref{mult_def} with the norm on the left-hand side replaced by sums over all indices:
\begin{equation} \label{mult_norm}
\left\|  \left(  T  x_{\mathbf{i}} \right)_{\mathbf{i} \in\mathbb{N}^{m}} \right\|_{\ell_r(F)}
:= \left( 
      \sum_{\mathbf{i} \in\mathbb{N}^{m}} \left\| T x_{\mathbf{i}} \right\|^{r}
   \right)^{\frac{1}{r}}
\end{equation}

Another successful line of research considers replacing the norm in \eqref{mult_norm} (with sums over the full generalized matrix $\mathbb{N}^m$) by a mixed norm with multiple parameters $\mathbf{r}:=(r_1,\dots,r_m) \in [1,+\infty)^m$. Specifically, we define
\begin{equation} \label{anisotropic}
\left\|
  \left( T x_{\mathbf{i}} \right)_{\mathbf{i} \in\mathbb{N}^{m}}
\right\|_{\ell_{\mathbf{r}}(F)}
:= \left( \sum_{i_{1}=1}^{\infty}
     \left( \cdots 
       \left( \sum_{i_{m}=1}^{\infty} \left\| T x_{\mathbf{i}} \right\|^{r_{m}}
       \right)^{\frac{r_{m-1}}{r_{m}}} \cdots
     \right)^{\frac{r_{1}}{r_{2}}}
   \right)^{\frac{1}{r_{1}}},
\end{equation}
and in this case, the operator is called \emph{multiple} $(\mathbf{r}; \mathbf{p})$-summing. This approach has led to significant developments and applications, including extensions of the classical Hardy--Littlewood and Bohnenblust--Hille inequalities. For recent advances in this theory, see \cite{abps-jfa,abps-israel,aacnnpr,blaise,bbbb,dimant,paulino,reg,zal}. The framework with multiple parameters and the mixed norm in \eqref{anisotropic} is referred to as the \emph{anisotropic} case (sums over all indices), while the classical absolutely summing case \eqref{mult_def} (sums restricted to the diagonal) is called \emph{isotropic} (see \cite{botelho-jmaa20,botelho-tmna25}).

\subsection{Multilinear summing operators: unified approach}

In \cite{bpr2020} and \cite{popa-lma} it was independetly introduced a concept that encompasses the isotropic and anisotropic notions previously described, and, moreover, \emph{intermediate} cases (in some sense will discuss next) are also included. The crucial idea is to consider the sum (the strong $\ell_r$-norm of \eqref{mult_def}, \eqref{mult_norm} and \eqref{anisotropic}) taking indices over an arbitrary but fixed subset $\Lambda$ of $\mathbb{N}^m$. The precise approach is defined as follows. Given $\mathbf{r}, \mathbf{p} \in [1,+\infty)^m$ and $\Lambda \subset \mathbb{N}^{m}$ a non-void set of indices, an $m$-linear operator $T: E_1 \times \cdots \times E_m \to F$ is $\Lambda$-$(\mathbf{r}; \mathbf{s})$-summing if there is a constant $C>0$ such that
\begin{equation}  \label{lambda-def}
\left\| \left( T x_\mathbf{i} \right)_{\mathbf{i} \in \Lambda}  \right\|_{\ell_{\mathbf{r}} \left( F \right)}
\leq C \prod_{j=1}^{m} \|(x_i^{j})_{i \in \mathbb{N}} \|_{w,p_{k}},
\end{equation}
for all $(x_{i}^{j})_{i\in\mathbb{N}} \in\ell_{p_{j}}^{w}(E_{j}), \, j=1,\dots,m$. The \(\ell_{\mathbf{r}}\)-norm in the left can be seen as
\[
\left\| \left( T x_\mathbf{i} \right)_{\mathbf{i} \in \Lambda} \right\|_{\ell_\mathbf{r}(E)}
=
\left\| \left( T x_\mathbf{i} \cdot 1_\Lambda (\mathbf{i}) \right)_{\mathbf{i} \in \mathbb{N}^m} \right\|_{\ell_\mathbf{r} (E)},
\]
where $1_\Lambda$ is the characteristic function of $\Lambda$. The class of all operators that fulfills \eqref{lambda-def} the previous inequality is denoted by $\Pi_{(\mathbf{r};\mathbf{s})}^{\Lambda} \left( E_1,\dots,E_m;F \right)$, which is a Banach space endowed with the norm $\pi_{(\mathbf{r};\mathbf{s})}^\Lambda( \cdot )$ taken as the infimum of the constants $C>0$ satisfying \eqref{lambda-def}. Notice that, by taking $\Lambda = \text{Diag}\left( \mathbb{N}^m \right) := \left\{ \left( n,\cdots,n \right) \in \mathbb{N}^m : n \in \mathbb{N} \right\}$ and $\Lambda = \mathbb{N}^{m}$, the $\Lambda$-summing class $\Pi^\Lambda$ recovers both absolutely summing class $\Pi^{\textrm{as}}$, and multiple summing class $\Pi^{\textrm{ms}}$, respectively. It is worth noting that for $\Lambda \subset \Gamma \subset \mathbb{N}^{m}$, the inclusions $\Pi^{\textrm{ms}} \subset \Pi^\Gamma \subset \Pi^\Lambda \subset \Pi^{\textrm{as}}$ and norm inequalities $\pi^{\textrm{ms}}(\cdot) \leq \pi^\Gamma(\cdot) \leq \pi^\Lambda(\cdot) \leq \pi^{\textrm{as}} (\cdot)$ clearly follows.

The study of \(\Lambda\)-summing operators, where \(\Lambda\) is an arbitrary non-empty set of indices, can be a challenging problem. A primary difficulty lies in computing the norm on the left-hand side of \eqref{lambda-def}. One approach to address this is by introducing a well-behaved \emph{block structure} on \(\Lambda\), as outlined below.

For $i_{j}\in \mathbb{N}$, $j\in\{1,\ldots,m\}$, we define
\[
i_{j}\cdot e_{j}:=(0,\ldots,0,i_{j},0,\ldots,0)\in \mathbb{N}^m,
\]
with $i_j$ in the $j$-th coordinate. For $1\leq d\leq m$ and $\mathcal{I} := \{I_{1},\dots,I_{d}\}$ a partition of non-void disjoints subsets of $\{1,\ldots,m\}$ such that $\cup_{i=1}^{d}I_{i}=\{1,\ldots,m\}$, the set of index
\[
\Lambda = \mathcal{B}_{\mathcal{I}} := \left\{  \sum_{n=1}^{d}\sum_{j\in I_{n}}i_{n}\cdot e_{j}\ :\ i_{1},\ldots,i_{d}\in\mathbb{N}\right\}  \subseteq\mathbb{N}^{m},
\]
is called a \emph{block of $\mathcal{I}$-type}. Notice that, for an arbitrary (non-void) indices set $\Lambda$, there exists a positive integer $t$, and blocks $\mathcal{B}_{\mathcal{I}^{(j)}}$ of $\mathcal{I}^{(j)}$-type, $j=1,\dots,t$, such that
\[
\Lambda = \mathcal{B}_{\mathcal{I}^{(1)}} \cup \cdots \cup \mathcal{B}_{\mathcal{I}^{(t)}}.
\]

We focus our attention for the block structure, that is, when $\Lambda= \mathcal{B}_{\mathcal{I}}$ is a block of $\mathcal{I}$-type. Given Banach spaces $E_{1},\ldots,E_{m}$ and $x_{j}\in E_{j}$, for some $j\in\{1,\ldots,m\}$, we defined
\[
x_{j}\cdot e_{j}:=(0,\ldots,0,x_{j},0,\ldots,0)\in E_{1}\times\cdots\times E_{m},
\]
that is, $x_{j}\cdot e_{j}$ is the element of $E_{1}\times\cdots\times E_{m}$ with $x_j$ in the $j$-th coordinate and $0$ elsewhere. The expression
\begin{equation*} 
\sum_{n=1}^{d}\sum\limits_{j\in I_{n}}x_{i_{n}}\cdot e_{j} \in E_{1}\times\cdots\times E_{m}
\end{equation*}
will be decisive throughout this matter. The next example clarifies the notation. 
\begin{example}
If
\[
m=5,\ d=3,\ \mathcal{I}_3=\{I_{1},I_{2},I_{3}\}, \ \text{ with } I_{1}=\{1,3\},\ I_{2}=\{2,4\}, \text{ and } I_{3}=\{5\},
\]
we have
\begin{align*}
\sum_{n=1}^{3}\sum_{j\in I_{n}}x_{i_{n}}\cdot e_{j}  &  =\sum_{j\in I_{1}}x_{i_{1}}\cdot e_{j}+\sum_{j\in I_{2}}x_{i_{2}}\cdot e_{j}+\sum_{j\in I_{3}}x_{i_{3}}\cdot e_{j}\\
&  =x_{i_{1}}\cdot e_{1}+x_{i_{1}}\cdot e_{3}+\sum_{j\in I_{2}}x_{i_{2}}\cdot e_{j}+\sum_{j\in I_{3}}x_{i_{3}}\cdot e_{j}\\
&=(x_{i_{1}},0,0,0,0)+(0,0,x_{i_{1}},0,0)+\sum_{j\in I_{2}}x_{i_{2}}\cdot e_{j}+\sum_{j\in I_{3}}x_{i_{3}}\cdot e_{j}\\
&=(x_{i_{1}},x_{i_{2}},x_{i_{1}},x_{i_{2}},x_{i_{3}}).
\end{align*}
\end{example}

Thus, as mentioned earlier, by considering the index set as a block type, the norm on the left side of \eqref{lambda-def} is more clearly computed. For instance, if $\Lambda = \mathcal{B}_{\mathcal{I}}$ is a block of $\mathcal{I}$-type, we have
\[
T x_{\mathbf{i}} = T \left(\sum_{n=1}^{d} \sum_{j\in I_{n}} x_{i_{n}}^{j} \cdot e_{j}\right),
\quad \text{ for } \mathbf{i} \in \mathcal{B}_{\mathcal{I}},
\]
and its $\ell_{\mathbf{r}}$-norm is written as
\begin{align*}
\left\| \left( T x_{\mathbf{i}} \right)_{\mathbf{i} \in \mathcal{B}_{\mathcal{I}}} \right\|_{\ell_{\mathbf{r}} \left( F \right)} 
&=\left( \sum_{i_{1}=1}^{\infty}\left( \cdots\left( \sum_{i_{d}=1}^{\infty}\left\| T \left(  \sum_{n=1}^{d}\sum_{j\in I_{n}} x_{i_{n}}^{j} \cdot e_{j}\right) \right\|_{F}^{r_{d}}\right)^{\frac{r_{d-1}}{r_{d}}}\cdots\right)^{\frac{r_{1}}{r_{2}}}\right)^{\frac{1}{r_{1}}}.
\end{align*}

We will simple write
\( \Pi^{\mathcal{B}_{\mathcal{I}}} := \Pi^{\Lambda} \) and \(\pi^{\mathcal{B}_{\mathcal{I}}} := \pi^{\Lambda}\). Note that by considering partitions $\mathcal{I}_{\mathrm{as}}:=\{\{1,\dots,m\}\}$ and $\mathcal{I}_{\mathrm{ms}}:=\{\{1\},\dots,\{m\}\}$, we recover the absolute and multiple summing classes.

As previously mentioned, the class \(\Pi^\Lambda\) (including its specific block case) possesses natural properties analogous to those of classical classes. To clarify, we introduce a condition to prevent cases where the class \(\Pi^{\mathcal{B}_{\mathcal{I}}}\) is trivial. Naturally, the parameters involved must satisfy a condition that integrates the requirements associated with absolutely and multiple summing classes. Thus argument is classical thus we omit the proof.

\begin{proposition}
Let $E_{1}, \dots, E_{m}, F$ be Banach spaces and let $m,d$ be positive integers with $1\leq d\leq m$, and $(\mathbf{r},\mathbf{p}) := (r_{1}, \ldots, r_{d}; p_{1}, \ldots, p_{m}) \in \lbrack 1,\infty)^{d+m}$. Let also $\mathcal{I}_d = \{I_{1}, \dots, I_{d}\}$ be a partition of $\{1,\ldots,m\}$ and $\mathcal{B}_{\mathcal{I}}$ a block of $\mathcal{I}$-type. If there exists $k \in \{1,\dots,d\}$ such that ${1}/{r_{k}}>\sum_{j\in I_{k}}{1}/{p_{j}}$, then
\[
\Pi_{\left( \mathbf{r}; \mathbf{p} \right)}^{\mathcal{B}_{\mathcal{I}}} (E_{1}, \dots, E_{m}; F) = \{0\}.
\]
\end{proposition}


For a more in-depth discussion on \(\Lambda\)-summing operators, we refer the reader to \cite{aacnnpr,bpr2020,popa-lma}. For a comprehensive study of the block scenario, we recommend the excellent works \cite{botelho-jmaa20,botelho-tmna25}.

\section{Inclusion theorem for block summing classes} \label{sec3}

For linear operators the classical inclusion theorem it is well known: let $s\geq r,\,q\geq p$ be such that
\(
\frac{1}{p} -\frac{1}{r}\leq\frac{1}{q}-\frac{1}{s},
\)
then every absolutely $(r;p)$-summing linear operator is absolutely $(s;q)$-summing. The linear case has the following extension for absolutely multilinear operators (see \cite{matos}):

\begin{theorem}
\label{matos_inclusao} Let $m$ be a positive integer, $1\leq r\leq s<\infty$
and $\mathbf{p},\mathbf{q}\in\lbrack1,\infty)^{m}$ be such that $q_{k}\geq
p_{k}$, for $k=1,\dots,m$. Then
\[
\Pi_{(r;\mathbf{p})}^{\mathrm{as}}\left(  E_{1},\dots,E_{m};F\right)
\subset\Pi_{(s;\mathbf{q})}^{\mathrm{as}}\left(  E_{1},\dots,E_{m};F\right)
,
\]
for any Banach spaces $E_{1},\dots,E_{m},F$, with
\[
\frac{1}{s}-\left\vert \frac{1}{\mathbf{q}}\right\vert =\frac{1}{r}-\left\vert
\frac{1}{\mathbf{p}}\right\vert ,
\]
and the inclusion operator has norm $1$.
\end{theorem}

Inclusion theorems become more intricate when dealing with multiple summing operators (see, e.g., \cite{bjp,dav,pe33}). Recently, this topic has been explored by various authors using different techniques (see \cite[Theorem 3]{AR}, \cite[Theorem 1.2]{bbbb}, \cite[Theorem 3]{paulino}, and \cite[Proposition 3.3]{reg}). We now present our main result, which generalizes all the aforementioned ones. Before proceeding, we introduce some notation. The conjugate of \( p \in (1,\infty) \) is denoted by \( p^{\ast} \), where \( \frac{1}{p}+\frac{1}{p^{\ast}}=1 \), with the convention that 1 and \( \infty \) are conjugates of each other. For \( A \subset \{1,\ldots,m\} \) and \( p_1,\dots,p_m \in [1,\infty] \), we define  
\[
\left| \frac{1}{\mathbf{p}} \right|_{j\in A} := \sum_{j\in A}\frac{1}{p_j}.
\]  
For \( 1\leq k\leq m \), we set \( |1/\mathbf{p}|_{j\geq k}:=|1/\mathbf{p}|_{j\in \{k,\ldots,m\}} \), and we write \( |1/\mathbf{p}| \) as a shorthand for \( |1/\mathbf{p}|_{j\geq 1} \).


\begin{theorem} \label{inclusion_blocks}
Let $1\leq d\leq m$ be positive integers and $\mathcal{I} = \{I_{1}, \dots, I_{d}\}$ be a partition of $\{1,\dots,m\}$ and and $\mathcal{B}_{\mathcal{I}}$ a block of $\mathcal{I}$-type. Let also $r\geq1,\,\mathbf{p},\mathbf{q}\in\lbrack1,\infty)^{m}$ and $r\leq s_{d}\leq\dots\leq s_{2}\leq s_1$ such that
\[
\frac{1}{s_{k}}-\left\vert \frac{1}{\mathbf{q}}\right\vert _{j\in \bigcup_{i=k}^{d}I_{i}}
=\frac{1}{r}-\left\vert \frac{1}{\mathbf{p}}\right\vert_{j\in\bigcup_{i=k}^{d}I_{i}},\ k=1,\ldots,d.
\]
If some of the following conditions holds,
\begin{itemize}
\item[(A)] $q_{j}\geq p_{j},\ j=1,\dots,m$, and
\[
\frac{1}{r} - \left\vert \frac{1}{\mathbf{p}}\right\vert +\left\vert \frac{1}{\mathbf{q}}\right\vert >0;
\]
\item[(B)] $q_{1}>p_{1},\ q_{j}\geq p_{j},\ j=2,\dots,m$, and
\[
\frac{1}{r}-\left\vert \frac{1}{\mathbf{p}}\right\vert +\left\vert \frac{1}{\mathbf{q}}\right\vert =0.
\]
\end{itemize}
then 
\[
\Pi_{\left(r; \mathbf{p} \right)}^{\mathcal{B}_{\mathcal{I}}} \left(E_{1}, \dots, E_{m}; F\right)
\subset
\Pi_{\left(\mathbf{s}; \mathbf{q}\right)}^{\mathcal{B}_{\mathcal{I}}} \left( E_{1}, \dots, E_{m}; F\right)
\]
for any Banach spaces $E_{1},\dots,E_{m},F$. Moreover, the inclusion operator has norm $1$.
\end{theorem}

\begin{proof}
%
Let us suppose that condition (A) holds. We proceed by induction on $d$. The bilinear case is a straightforward application of Theorems \ref{matos_inclusao} and \cite[Theorem 3]{AR}.
	Let us suppose the result is true for all for $d-1$. Let $T\in\Pi_{\left(  r;\mathbf{p}
		\right)  }^{\mathcal{B}_{\mathcal{I}}}\left(  E_{1},\dots,E_{m};F\right)  $, i.e.,
	\begin{align*}
	&\left(  \sum_{j_{1}=1}^{n_{1}}\left(  \ldots\left(  \sum_{j_{d}=1}^{n_{d}}\left\Vert T\left(  \sum_{k=1}^{d}\sum_{j\in J_{k}}x_{j_{k}}^{j}\cdot e_{j}\right)\right\Vert ^{r}\right)  ^{\frac{r}{r}}\dots\right)  ^{\frac{r}{r}}\right)^{\frac{1}{r}} \\
	& =\left(  \sum_{j_{1},\dots,j_{d}=1}^{\infty}\left\Vert T\left(  \sum_{k=1}^{d}\sum_{j\in J_{k}}x_{j_{k}}^{j}\cdot e_{j}\right)  \right\Vert^{r}\right)  ^{\frac{1}{r}}\leq C\cdot\prod_{k=1}^{m}\left\Vert x^{k}\right\Vert _{w,p_{k}},
	\end{align*}
	for all $x^{k}\in\ell_{p_{k}}^{w}(E_{k})$. Without loss of
	generality, suppose that $I_{1}=\{1,\dots,l\}$. Thus, $\mathcal{J}:=\{I_{2},\dots,I_{d}\}$ is a partition of $\{l+1,\dots,m\}$. Fixed
	$x^{k}\in\ell_{p_{k}}^{w}(E_{k}),k=1,\ldots,l$, let us define $w:E_{l+1}\times
	\cdots\times E_{m}\rightarrow\ell_{r}(F)$ given by
	\[
	w(x_{l+1},\ldots,x_{m}):=\left(  T(x_{j}^{1},\dots,x_{j}^{l},x_{l+1}%
	,\ldots,x_{m})\right)  _{j\in\mathbb{N}}.
	\]
Observe that $w$ belongs to $\Pi_{(r;p_{l+1},\dots,p_{m})}^{\mathcal{B}_{\mathcal{J}}}\left(
E_{l+1},\dots,E_{m};\ell_{r}(F)\right)$. Consequently, by the induction hypothesis, \(\ell_p\) norm inclusions, and Minkowski's inequality,
	\begin{align*}
	&  \left(  \sum_{j_{1}=1}^{\infty}\left(  \sum_{j_{2}=1}^{\infty}\left(
	\ldots\left(  \sum_{j_{d}=1}^{\infty}\left\Vert T\left(  \sum_{k=1}^{d}%
	\sum_{j\in I_{k}}x_{j_{k}}^{j}\cdot e_{j}\right)  \right\Vert ^{s_{d}}\right)
	^{\frac{s_{d-1}}{s_{d}}}\ldots\right)  ^{\frac{s_{2}}{s_{3}}}\right)
	^{\frac{s_{2}}{s_{2}}}\right)  ^{\frac{1}{s_{2}}}\\
	&  \leq\left(  \sum_{j_{2}=1}^{\infty}\left(  \ldots\left(  \sum_{j_{d}%
		=1}^{\infty}\left(  \sum_{j_{1}=1}^{\infty}\left\Vert T\left(  \sum_{j\in
		I_{1}}x_{j_{1}}^{j}\cdot e_{j}+\sum_{k=2}^{d}\sum_{j\in I_{k}}x_{j_{k}}^{j}%
	\cdot e_{j}\right)  \right\Vert ^{r}\right)  ^{\frac{s_{d}}{r}}\right)
	^{\frac{s_{d-1}}{s_{d}}}\ldots\right)  ^{\frac{s_{2}}{s_{3}}}\right)
	^{\frac{1}{s_{2}}}\\
	&  =\left(  \sum_{j_{2}=1}^{\infty}\left(  \ldots\left(  \sum_{j_{d}%
		=1}^{\infty}\left\Vert w\left(  \sum_{k=2}^{d}\sum_{j\in I_{k}}x_{j_{k}}%
	^{j}\cdot e_{j}\right)  \right\Vert _{r}^{s_{d}}\right)  ^{\frac{s_{d-1}}{s_{d}}%
	}\ldots\right)  ^{\frac{s_{2}}{s_{3}}}\right)  ^{\frac{1}{s_{2}}}\\
	&  \leq C\cdot\prod_{k=1}^{l}\left\Vert x^{k}\right\Vert _{w,p_{k}}\cdot
	\prod_{k=l+1}^{m}\left\Vert x^{k}\right\Vert _{w,q_{k}},
	\end{align*}
	with $r\leq s_{d}\leq\dots\leq s_{2}$ and
	\[
	\frac{1}{s_{k}}=\frac{1}{r}-\left\vert \frac{1}{\mathbf{p}}\right\vert
	_{j\in\bigcup_{i=k}^{d}I_{i}}+\left\vert \frac{1}{\mathbf{q}}\right\vert
	_{j\in\bigcup_{i=k}^{d}I_{i}},
	\quad \text{ for each } \quad k=2,\ldots,d.
	\]
	Now fixed $x^{k}\in\ell_{p_{k}}^{w} (E_{k}),k=l+1,\ldots,m$, let us define, for all $(x_{1},\dots,x_{l})\in
	E_{1}\times\cdots\times E_{l}$,
	\[
	\psi(x_{1},\ldots,x_{l}):=\left(  T\left(  \sum_{j\in I_{1}}x_{j}\cdot e_{j}%
	+\sum_{k=2}^{d}\sum_{j\in I_{k}}x_{j_{k}}^{j}\cdot e_{j}\right)  \right)
	_{j_{2},\dots,j_{d}\in\mathbb{N}}.
	\]
	Notice that
	\[
	\left(  \sum_{j_{1}=1}^{\infty}\left\Vert \psi\left(  x_{j_{1}}^{1}%
	,\ldots,x_{j_{1}}^{l}\right)  \right\Vert _{\ell_{(s_{2},\dots,s_{d})}}%
	^{s_{2}}\right)  ^{\frac{1}{s_{2}}}\leq C_{1}\cdot\prod_{k=1}^{l}\left\Vert
	x^{k}\right\Vert _{w,p_{k}},
	\]
	where $C_{1}=C\cdot\prod_{k=l+1}^{m}\left\Vert x^{k}\right\Vert _{w,q_{k}}$,
	i.e., $\psi\in\Pi_{(s_{2};p_{1},\dots,p_{l})}^{\mathrm{as}}\left(  E_{1}%
	,\dots,E_{l};\ell_{s_{2},\dots,s_{d}}(F)\right)  $. From Theorem
	\ref{matos_inclusao} we conclude that $\psi\in\Pi_{(s_{1};q_{1},\dots,q_{l}%
		)}^{as}\left(  E_{1},\dots,E_{l};\ell_{s_{2},\dots,s_{d}}(F)\right)  $, with
	\[
	\frac{1}{s_{1}}=\frac{1}{s_{2}}-\sum_{j\in I_{1}}\frac{1}{p_{j}}+\sum_{j\in
		I_{1}}\frac{1}{q_{j}}=\frac{1}{r}-\left\vert \frac{1}{\mathbf{p}}\right\vert
	_{j\in\bigcup_{i=1}^{d}I_{i}}+\left\vert \frac{1}{\mathbf{q}}\right\vert
	_{j\in\bigcup_{i=1}^{d}I_{i}}.
	\]

\bigskip

Now we deal with the hypothesis (B). Let $T\in\Pi_{\left(  r;\mathbf{p}\right) }^{\mathcal{B}_{\mathcal{I}}} \left(  E_{1},\dots,E_{m};F\right)  $, i.e.,
	\begin{align*}
	& \left(  \sum_{j_{1}=1}^{n_{1}}\left(  \ldots\left(  \sum_{j_{d}=1}^{n_{d}}\left\Vert T\left(  \sum_{k=1}^{d}\sum_{j\in J_{k}}x_{j_{k}}^{j}\cdot e_{j}\right)\right\Vert ^{r}\right)  ^{\frac{r}{r}}\dots\right)  ^{\frac{r}{r}}\right)^{\frac{1}{r}} \\
	&  = \left(  \sum_{j_{1},\dots,j_{d}=1}^{\infty}\left\Vert T\left(  \sum_{k=1}^{d}\sum_{j\in J_{k}}x_{j_{k}}^{j}\cdot e_{j}\right)  \right\Vert^{r}\right)  ^{\frac{1}{r}} \leq C\cdot\prod_{k=1}^{m}\left\Vert x^{k}\right\Vert _{w,p_{k}},
	\end{align*}
	for all sequences $x^{k}\in\ell_{p_{k}}^{w}(E_{k})$. From $q_{1}>p_{1}$ it follows that
	\[
	\frac{1}{r}-\left\vert \frac{1}{\mathbf{p}}\right\vert _{\geq2}+\left\vert \frac{1}{\mathbf{q}}\right\vert _{\geq2}>0.
	\]
For the sake of clarity, we suppose that $I_{1}=\{1\}$. Then $\mathcal{K}:=\{I_{2},\dots,I_{d}\}$ is a partition of $\{2,\dots,m\}$. Fixed $x^{1}\in\ell_{p_{1}}^{w}(E_{1})$,
define $\beta:E_{2}\times\cdots\times E_{m}\rightarrow\ell_{r}(F)$ given by
\[
\beta(x_{2},\ldots,x_{m}):=\left(  T(x_{j}^{1},x_{2},\ldots,x_{m})\right)
_{j\in\mathbb{N}}.
\]
Thus, \(\beta \in \Pi_{(r;p_{2},\dots,p_{m})}^{\mathcal{B}_{\mathcal{K}}} (E_{2},\dots,E_{m};\ell_{r}(F))\). We apply the result established in case (A), proceeding accordingly:
\begin{align*}
	&  \left(  \sum_{j_{1}=1}^{\infty}\left(  \sum_{j_{2}=1}^{\infty}\left(
	\ldots\left(  \sum_{j_{d}=1}^{\infty}\left\Vert T\left(  \sum_{k=1}^{d}%
	\sum_{j\in I_{k}}x_{j_{k}}^{j}\cdot e_{j}\right)  \right\Vert ^{s_{d}}\right)
	^{\frac{s_{d-1}}{s_{d}}}\ldots\right)  ^{\frac{s_{2}}{s_{3}}}\right)
	^{\frac{s_{2}}{s_{2}}}\right)  ^{\frac{1}{s_{2}}}\\
	&  \leq\left(  \sum_{j_{2}=1}^{\infty}\left(  \ldots\left(  \sum_{j_{d}%
		=1}^{\infty}\left(  \sum_{j_{1}=1}^{\infty}\left\Vert T\left(  x_{j_{1}}%
	^{1}\cdot e_{1}+\sum_{k=2}^{d}\sum_{j\in I_{k}}x_{j_{k}}^{j}\cdot e_{j}\right)
	\right\Vert ^{r}\right)  ^{\frac{s_{d}}{r}}\right)  ^{\frac{s_{d-1}}{s_{d}}%
	}\ldots\right)  ^{\frac{s_{2}}{s_{3}}}\right)  ^{\frac{1}{s_{2}}}\\
	&  =\left(  \sum_{j_{2}=1}^{\infty}\left(  \ldots\left(  \sum_{j_{d}%
		=1}^{\infty}\left\Vert \beta\left(  \sum_{k=2}^{d}\sum_{j\in I_{k}}x_{j_{k}%
	}^{j}\cdot e_{j}\right)  \right\Vert _{r}^{s_{d}}\right)  ^{\frac{s_{d-1}}{s_{d}}%
	}\ldots\right)  ^{\frac{s_{2}}{s_{3}}}\right)  ^{\frac{1}{s_{2}}}\\
	&  \leq C\cdot\left\Vert x^{1}\right\Vert _{w,p_{1}}\cdot\prod_{k=2}%
	^{m}\left\Vert x^{k}\right\Vert _{w,q_{k}},
	\end{align*}
	with $r\leq s_{d}\leq\dots\leq s_{2}$ and
	\[
	\frac{1}{s_{k}}=\frac{1}{r}-\left\vert \frac{1}{\mathbf{p}}\right\vert
	_{j\in\bigcup_{i=k}^{d}I_{i}}+\left\vert \frac{1}{\mathbf{q}}\right\vert
	_{j\in\bigcup_{i=k}^{d}I_{i}},
	\quad \text{ for each } \quad k=2,\ldots,d.
	\]
	Now we fix $x^{k}\in\ell_{p_{k}}^{w} (E_{k}),k=2,\ldots,m$, and define, for all $x\in E_{1}$,
	\[
	\xi(x):=\left(  T\left(  x\cdot  e_{1}+\sum_{k=2}^{d}\sum_{j\in I_{k}%
	}x_{j_{k}}^{j}\cdot e_{j}\right)  \right)  _{j_{2},\dots,j_{d}\in\mathbb{N}}.
	\]
Then
\[
\left(  \sum_{j_{1}=1}^{\infty}\left\Vert \xi\left(  x_{j_{1}}^{1}\right)
\right\Vert _{\ell_{(s_{2},\dots,s_{d})}}^{s_{2}}\right)  ^{\frac{1}{s_{2}}%
}\leq C_{1}\cdot\left\Vert x^{1}\right\Vert _{w,p_{1}},
\]
where $C_{1}=C\cdot\prod_{k=2}^{m}\left\Vert x^{k}\right\Vert _{w,q_{k}}$, i.e., $\xi\in\Pi_{(s_{2};p_{1})}\left(  E_{1};\ell_{s_{2},\dots,s_{d}
}(F)\right)$. Applying the Classical Inclusion Theorem, we obtain \(\xi \in \Pi_{(s_{1};q_{1})} (E_{1};\ell_{s_{2},\dots,s_{d}}(F))\), where
\[
\frac{1}{s_{1}}
=\frac{1}{s_{2}} - \frac{1}{p_{1}} + \frac{1}{q_{1}}
=\frac{1} {r}-\left\vert \frac{1}{\mathbf{p}}\right\vert_{j \in \bigcup_{i=1}^{d}I_{i}} + \left\vert \frac{1}{\mathbf{q}}\right\vert_{j \in \bigcup_{i=1}^{d}I_{i}}.
\]
This completes the proof.
\end{proof}

\section{Applications} \label{sec4}

In this section, \( e_{i}^{n} \) denotes the \( n \)-tuple \( (e_{i},\dots,e_{i}) \), where \( e_{i} \) is the canonical vector of the sequence space \( c_{0} \). We define \( X_{p} = \ell_{p} \) for \( 1 \leq p < \infty \) and \( X_{\infty} = c_{0} \).

\subsection{Hardy--Littlewood block variants inequalities}

In 1934, G. Hardy and E. Littlewood \cite{hardy} extended Littlewood's $4/3$ inequality to bilinear forms on $\ell_{p} \times \ell_{q}$. In 1981, T. Praciano-Pereira \cite{pra} generalized the Hardy–Littlewood inequalities to $m$-linear forms on $\ell_{p_{1}} \times \cdots \times \ell_{p_{m}}$ for $0 \leq \frac{1}{p_{1}} + \cdots + \frac{1}{p_{m}} \leq \frac{1}{2}$. Later, V. Dimant and P. Sevilla-Peris \cite{dimant} further extended these inequalities to the case $\frac{1}{2} \leq \frac{1}{p_{1}} + \cdots + \frac{1}{p_{m}} < 1$. The following result combines the Hardy–Littlewood and Dimant–Sevilla-Peris inequalities.

\begin{theorem}[Hardy--Littlewood/Dimant--Sevilla-Peris \cite{dimant,hardy}]
Let $\mathbf{p}\in[1,\infty)^{m}$ such that $\frac{1}{2}\leq\left\vert\frac{1}{\mathbf{p}}\right\vert <1$. Then, for every continuous $m$--linear forms $T:X_{p_{1}} \times\cdots\times X_{p_{m}} \to\mathbb{K}$, there exists a constant $C_{m,\mathbf{p}}^\mathbb{K}\geq 1$ such that
\[
\left(\sum_{i_{1},\dots,i_{m}=1}^{\infty}\left\vert T\left(  e_{i_{1}},\dots,e_{i_{m}}\right)  \right\vert ^{\frac{1}{1-\left\vert\frac{1}{\mathbf{p}}\right\vert}}\right)  ^{1-\left\vert \frac{1}{\mathbf{p}}\right\vert }\leq C_{m,\mathbf{p}}^{\mathbb{K}}\left\Vert T\right\Vert
\]
and the exponent is optimal.
\end{theorem}

In \cite[Theorem 2.4]{aacnnpr} the following version with blocks of the Hardy--Littlewood/Dimant--Sevilla-Peris inequalities was proven.

\begin{theorem}
	\label{hl_rewritten} Let $\mathbf{p} \in[1,\infty)^{m}$ and let $1\leq d \leq m,\ n_{1},\ldots,n_{d}$ positive integers such that $n_{1}+\cdots+n_{d} =m$. If $\frac{1}{2}\leq\left\vert\frac{1}{\mathbf{p}}\right\vert <1$, then, for every continuous $m$--linear forms $T:X_{p_{1}} \times\cdots\times X_{p_{m}} \to\mathbb{K}$,
	we have
	\[
	\left(  \sum_{i_{1},\dots,i_{d}=1}^{\infty}\left\vert T\left(  e_{i_{1}%
	}^{n_{1}},\dots,e_{i_{d}}^{n_{d}}\right)  \right\vert ^{\frac{1}{1-\left\vert
			\frac{1}{\mathbf{p}}\right\vert }}\right)  ^{1-\left\vert \frac{1}{\mathbf{p}%
		}\right\vert }\leq C_{k,\mathbf{p}}^{\mathbb{K}}\left\Vert T\right\Vert 
	\]
	for some constant $C_{k,\mathbf{p}}^\mathbb{K}\geq 1$.
	Moreover, the exponent is optimal.
\end{theorem}

Although the exponents in the above theorems are sharp, these results can be improved when considering anisotropic exponents (see \cite[Theorem 3.4 and Theorem 3.5]{blaise} and \cite[Corollary 2]{AR}). Our Inclusion Theorem has a direct application to the study of Hardy–Littlewood inequalities for multilinear forms. Specifically, we will use Theorem \ref{inclusion_blocks} to strengthen the aforementioned results for the case where \( p_{1}, \dots, p_{m} \leq 2m \).

\begin{theorem}
	\label{new_HL_blocks} Let $m$ be a positive integer, $\mathcal{I}%
	=\{I_{1},\dots,I_{d}\}$ a partition of $\{1,\dots,m\}$ and $\mathbf{p}\in (1,2m]^{m}$ such that $\left\vert 1/\mathbf{p}\right\vert <1$. Then, there exists a constant $D_{m,\mathbf{p},\mathbf{s}}^{\mathbb{K}}\geq1$ such that
	\begin{equation}
	\left(  \sum_{j_{1}=1}^{n}\left(  \cdots\left(  \sum_{j_{d}=1}^{n}\left\vert A\left(  \sum_{n=1}^{d}\sum_{j\in I_{n}}e_{i_{n}}\cdot e_{j}\right) \right\vert ^{s_{d}}\right)  ^{\frac{s_{d-1}}{s_{d}}}\cdots\right) ^{\frac{s_{1}}{s_{2}}}\right)  ^{\frac{1}{s_{1}}}\leq D_{m,\mathbf{p},\mathbf{s}}^{\mathbb{K}}\Vert A\Vert, \label{new_exp}
	\end{equation}
	for all positive integers $n$ and all $m$-linear forms $A:\ell_{p_{1}}%
	^{n}\times\cdots\times\ell_{p_{m}}^{n}\rightarrow\mathbb{K}$, with
	\[
	s_{k}=\left[  \frac{1}{2}-\left\vert \frac{1}{\mathbf{p}}\right\vert
	_{j\in\bigcup_{i=k}^{d}I_{i}}+\frac{1}{2m}\cdot\sum_{i=k}^{d}|I_{i}|\right]
	^{-1},\quad\text{ for }k=1,\dots,d.
	\]
	
\end{theorem}

\begin{proof}
	Let $\mathbf{2m}:=(2m,\ldots,2m)$, $(\mathbf{2m})^\ast:=((2m)^\ast,\ldots,(2m)^\ast)$ and $\mathbf{p}^\ast:=(p_1^\ast,\ldots,p_m^\ast)$. Let also $r:=2=\left(  1-\left\vert 1/\mathbf{2m}\right\vert \right)  ^{-1}$. Thus
	\[
	\frac{1}{2}-\left\vert \frac{1}{\mathbf{(2m)^{\ast}}}\right\vert +\left\vert
	\frac{1}{\mathbf{p^{\ast}}}\right\vert =\frac{1}{2}-m+\frac{1}{2}+m-\left\vert
	\frac{1}{\mathbf{p}}\right\vert =1-\left\vert \frac{1}{\mathbf{p}}\right\vert
	>0
	\]
	and, from Theorem \ref{inclusion_blocks}, we have
	\[
	\Pi_{\left(  2;(\mathbf{2m})^{\ast}\right)  }^{\mathcal{B}_{\mathcal{I}}}\left(
	E_{1},\dots,E_{m};\mathbb{K}\right)  \subset\Pi_{\left(  \mathbf{s};\mathbf{p}^{\ast}\right)  }^{\mathcal{B}_{\mathcal{I}}}\left(  E_{1}%
	,\dots,E_{m};\mathbb{K}\right)  ,
	\]
	for any Banach spaces $E_{1},\dots,E_{m}$, with
	\begin{align*}
	s_{k}  &  =\left[  \frac{1}{2}-\left\vert \frac{1}{\mathbf{(2m)^{\ast}}%
	}\right\vert _{j\in\bigcup_{i=k}^{d}I_{i}}+\left\vert \frac{1}{\mathbf{p^{\ast
	}}}\right\vert _{j\in\bigcup_{i=k}^{d}I_{i}}\right]  ^{-1}\\
	&  =\left[  \frac{1}{2}-\sum_{i=k}^{d}|I_{i}|+\frac{1}{2m}\cdot\sum_{i=k}%
	^{d}|I_{i}|+\sum_{i=k}^{d}|I_{i}|-\left\vert \frac{1}{\mathbf{p}}\right\vert
	_{j\in\bigcup_{i=k}^{d}I_{i}}\right]  ^{-1}\\
	&  =\left[  \frac{1}{2}-\left\vert \frac{1}{\mathbf{p}}\right\vert
	_{j\in\bigcup_{i=k}^{d}I_{i}}+\frac{1}{2m}\cdot\sum_{i=k}^{d}|I_{i}|\right]
	^{-1},\quad\text{ for each }k=1,\dots,d.
	\end{align*}
	
Now by using Theorem \ref{hl_rewritten},
\[
\Pi_{\left(  2;(\mathbf{2m})^{\ast}\right)  }^{\mathcal{B}_{\mathcal{I}}}\left(
E_{1},\dots,E_{m};\mathbb{K}\right)  =\mathcal{L}\left(  E_{1},\dots
,E_{m};\mathbb{K}\right)
\]
and thus
\[
\Pi_{\left(  \mathbf{s};\mathbf{p}^{\ast}\right)
}^{\mathcal{B}_{\mathcal{I}}}\left(  E_{1},\dots,E_{m};\mathbb{K}\right)  =\mathcal{L}%
\left(  E_{1},\dots,E_{m};\mathbb{K}\right)  ,
\]
for all Banach spaces $E_{1},\dots,E_{m}$. Using the standard isometries
between $\mathcal{L}(X_{p},X)$ and $\ell_{p^{\ast}}^{w}(X)$, for
$1<p\leq\infty$, the proof is completed.
\end{proof}

In order to clarify the new result, we illustrate the following simpler case
that provides better exponents than the estimates of Theorem
\ref{hl_rewritten}.

\begin{corollary}\label{corollary43}
	Let $1\leq d \leq m$ and let $n_{1},\ldots,n_{d}$ be positive integers such that
	$n_{1}+\cdots+n_{d} =m$. If $m<p\leq2m$, then
	\[
	\left(  \sum_{j_{1}=1}^{\infty}  \left(  \cdots\left(  \sum_{j_{d}=1}%
	^{\infty}  \left\vert A \left(  e_{i_{1}}^{n_{1}}, \ldots, e_{i_{d}}^{n_{d}}
	\right)  \right\vert ^{s_{d}}  \right) ^{\frac{s_{d-1}}{s_{d}}} \cdots\right)
	^{\frac{s_{1}}{s_{2}}} \right) ^{\frac{1}{s_{1}}} \leq D_{m,\mathbf{p}%
		,\mathbf{s}}^{\mathbb{K}} \|A\|,
	\]
	for all $m$-linear forms $A:\ell_{p}^{n}\times\cdots\times\ell_{p}^{n}\rightarrow\mathbb{K}$, with
	\[
	s_{k}=\left[  \frac{1}{2}-\left(  n_{k}+\cdots+n_{d} \right)  \cdot\left(
	\frac{1}{p}-\frac{1}{2m}\right)  \right] ^{-1}, \quad\text{ for } \quad
	k=1,\dots,d.
	\]
\end{corollary}

Notice that the exponents provided are
\[
s_{1}=\frac{p}{p-m},\dots,s_{d}=\frac{2mp}{mp+pn_{d}-2mn_{d}},
\]
while Theorem \ref{hl_rewritten} provide the exponents
\[
s_{1} = \cdots= s_{d} = \frac{p}{p-m}.
\]
Corollary \ref{corollary43} recovers \cite[Corollary 2]{AR} and, when $d=1$, recovers a result of \cite{zal}. The following example is intended to illustrate this.

\begin{example}
Suppose $m=3,d=2,p=4,I_{1}=\{1,2\}$ and $I_{2}=\{3\}$. By Theorem
\ref{hl_rewritten} we know that \eqref{new_exp} holds with
\(
s_{k}=4 \text{ for } k=1,2,
\)
whereas by Theorem \ref{new_HL_blocks} we have
\(
s_{1}=4 \text{ and } s_{2}=12/5.
\)\
\end{example}

\subsection{A new Grothendieck inclusion result}

Grothendieck's famous theorem for absolutely summing linear operators has been extended to the $m$-linear setting (see \cite[Theorems 5.1 and 5.2]{bombal} and \cite{dav}). More specifically, every continuous $m$-linear operator \( T \colon \ell_{1} \times \cdots \times \ell_{1} \to \ell_{2} \) is multiple $p$-summing for all \( p \in [1, 2] \). A natural question arises: for which pairs \((q, p)\) does the coincidence  
\[
\Pi_{(q; p)}^{\mathrm{ms}} \left( ^{m}\ell_{1}; \ell_{2} \right) = \mathcal{L} \left( ^{m}\ell_{1}; \ell_{2} \right)
\]  
hold? In \cite{reg} the authors present a definitive answer for this question.

\begin{theorem}\cite{reg}\label{dan}
Let $m$ be a positive integer and $1\leq p\leq q<\infty.$ Then
\(
\Pi_{(q;p)}^{\mathrm{ms}} \left(  ^{m}\ell_{1};\ell_{2}\right)  =\mathcal{L}\left(^{m}\ell_{1};\ell_{2}\right)
\)
if and only if $p\leq2$ or $q>p>2.$
\end{theorem}

Given the various generalizations of this theory to anisotropic settings, a natural question arises: For which multi-indices \((\mathbf{q}; \mathbf{p}) = (q_{1}, \ldots, q_{m}; p_{1}, \ldots, p_{m})\), with \(q_i \neq q_j\) or \(p_i \neq p_j\) for some \(i,j \in \{1, \ldots, m\}\), does the equality  
\[
\Pi_{(\mathbf{q}; \mathbf{p})}^{\mathrm{ms}} \left(^{m}\ell_{1}; \ell_{2}\right) = \mathcal{L} \left(^{m}\ell_{1}; \ell_{2}\right)
\]
hold? In \cite[Proposition 4.1(a)]{bot}, it was shown that  
\(
\Pi_{(q; 1, \ldots, 1, q)}^{\mathrm{ms}} \left(^{m}\ell_{1}; \ell_{2}\right)
= \mathcal{L} \left(^{m}\ell_{1}; \ell_{2}\right)
\)
for all \(q \geq 2\). As a consequence of Theorem \ref{inclusion_blocks}, we provide two additional solutions to this problem.

\begin{proposition} Let $m$ be a positive integer, $p\leq2$ or $r>p>2$ and  $\mathbf{s},\mathbf{q }\in[1,\infty)^{m}$. If $1/r - m/p  +\left|  1/\mathbf{q}\right|  > 0$, and $q_{k} \geq p$ and
\[
\frac{1}{s_{k}} - \left|  \frac{1}{\mathbf{q}}\right|  _{\geq k} = \frac{1}{r}
- \frac{m-k+1}{\mathbf{p}}
\]
for $k=1,\dots,m$, then
\[
\Pi_{(\mathbf{s};\mathbf{q})}^{\mathrm{ms}} \left(^{m}\ell_{1};\ell_{2}\right) = \mathcal{L}\left(^{m}\ell_{1};\ell_{2}\right).  
\]
\end{proposition}

\begin{proof}
By using Theorem \ref{inclusion_blocks}, with $\mathcal{I} = \{\{1\}, \ldots, \{m\}\}$, we get that
\[
\Pi_{(r;p)}^{m} \left(^{m}\ell_{1};\ell_{2}\right)  \subset\Pi_{(\mathbf{s};\mathbf{q})}^{m} \left(^{m}\ell_{1};\ell_{2}\right),
\]
with
\[
\frac{1}{s_{k}} - \left|  \frac{1}{\mathbf{q}}\right|  _{\geq k} = \frac{1}{r}-\frac{m-k+1}{\mathbf{p}}, \ \ \ k=1,\ldots,m.
\]
Combining this with Theorem \ref{dan}, the result follows.
\end{proof}

\vskip 10mm


\section*{Acknowledgement}

The first author was supported in part by CNPq Grants 312167/2021-0, 406457/2023-9, and 403964/2024-5. The last author was partially supported by CAPES.

\end{document}